\documentclass[12pt]{amsart}
\usepackage{amssymb}
\usepackage[latin1]{inputenc}
\usepackage{amsthm}
\usepackage{amsfonts}
\usepackage{mathrsfs}
\usepackage{graphicx}
\usepackage{amsmath}
\usepackage[all]{xy} 
\numberwithin{equation}{section}

\newtheorem{teo}{Theorem}[section] 
\newtheorem{lem}[teo]{Lemma}

\newtheorem{defi}[teo]{Definition}
\newtheorem{oss}[teo]{Remark}
\newtheorem{ex}[teo]{Example}
\newtheorem{cex}[teo]{Counterexample}

\newtheorem{conjecture}[teo]{Conjecture}

\title[Bounds on the denominators in the canonical bundle formula]{Bounds on the denominators in the canonical bundle formula}
\date{\today}

\author{Enrica Floris}

\address{Enrica Floris\\
IRMA, Universit\'e de Strasbourg et CNRS\\
7 rue Ren\'e-Descartes\\
67084 Strasbourg Cedex\\
France}
\email{floris@math.unistra.fr}

\begin{document}

\begin{abstract}
In this work we study the moduli part $M_Z$ in the canonical bundle formula
of an lc-trivial fibration $f\colon(X,B)\rightarrow Z$ whose generic fibre $F$ is a rational curve.
If $r$ is the Cartier index of $(F,B_F)$
it was expected that $12r$ would provide a bound on the denominators of $M_Z$.
Here we prove that such a bound cannot even be polynomial in $r$,
we provide a bound $N(r)$ and an example where the minimum integer
$V$ such that $VM_Z$ has integer coefficients is at least $N(r)/r$.
Moreover we prove that even locally the denominators of $M_Z$
depend quadratically on $r$.
\end{abstract}

\maketitle

\section{Introduction}
The canonical bundle formula is an important tool in classification theory
to reduce the study of varieties of intermediate Kodaira dimension,
that is $0<{\rm kod}(X)<\dim X$,
to the study of varieties, more precisely pairs,
having Kodaira dimension 0 or equal to their dimension.

To be precise, let $(X,B)$ be a log canonical pair,
where $X$ is a normal variety of dimension $n$ over the field $\mathbb{C}$
and $B$ a $\mathbb{Q}$-divisor.
We consider the canonical ring of $(X,B)$
$$R(X,B)=\oplus \Gamma(X,m(K_X+B))$$
where the sum runs over the $m$ sufficiently divisible.
If $R(X,B)$ is not the ring $0$, then for $m$ sufficiently large and divisible
$|m(K_X+B)|$ defines a morphism
$$\phi\colon X'\rightarrow Z$$
where $X'$ is some birational model of $X$.
There are three cases.
\begin{enumerate}
\item If $\dim Z=0$ then $K_{X'}+B'$ is torsion.
\item If $0<\dim Z<n$ then $\phi$ is a fibration with general fibre $F$
such that $K_F+B'|_F$ is torsion.
\item If $\dim Z=n$ then $(X,B)$ is \textit{of log general type}.
\end{enumerate}
If $X$ is a smooth surface and $B=0$ the three cases become the following.
\begin{enumerate}
\item The canonical divisor $K_X$ is torsion and more precisely
$mK_X\cong\mathcal{O}_X$ for some $m\in\{1,2,3,4,6\}$.
Smooth surfaces of this type are classified up to isomorphism.
\item The morphism $\phi$ is a fibration with generic fibre an elliptic curve.
\item If $\dim Z=2$ then $X$ is of general type.
\end{enumerate}
In the second case we have Kodaira's canonical bundle formula for a minimal elliptic surface
(see for instance ~\cite[Chapter V, Theorem 12.1]{BPV})
\begin{eqnarray}\label{kodsup}
K_X&=&\phi^{\ast}(K_Z+\sum_{p\in Z}(1-\frac{1}{m_p})p+L)
\end{eqnarray}
where $L$ is of the form $R+j^{\ast}\mathcal{O}_{\mathbb{P}^1}(1)$,
with $R$ is supported on the singular locus of $\phi$
and $j\colon Z\rightarrow\mathbb{P}^1$ is the $j$-function.
The sum in the formula is over the $p\in Z$ such that $\phi^{\ast}p$
is a multiple fibre and $m_p$ is such that
$\phi^{\ast}p=m_p S_p$ where $S_p$ is the support of the fibre.
Kawamata in ~\cite{Kaw1,Kaw2} pointed out that the divisor $R+\sum(1-1/m_p)p$
can be computed in terms of the pair $(X,B)$.
More precisely, if $R+\sum(1-1/m_p)p=\sum b_p p$
then $1-b_p$ is the largest real number $t$
such that the pair $(X,B+t f^{\ast}p)$ is log canonical.
In the case where $X$ has dimension $n$, the current generalization of the formula is
due to Ambro ~\cite{Amb04} and reads as follows:
\begin{eqnarray}\label{CBFamb}
K_X + B +\frac{1}{r}(\varphi) &= &\phi^{\ast}(K_Z+B_Z+M_Z)
\end{eqnarray}
where $r\in\mathbb{N}$ is the Cartier index of the fibre, $\varphi$ is a rational function,
the divisor $B_Z$ is called the \textit{discriminant} and corresponds to $\sum(1-\frac{1}{m_p})p+R$
in Kodaira's formula,
while $M_Z$, called the \textit{moduli part}, corresponds to $j^{\ast}\mathcal{O}_{\mathbb{P}^1}(1)$
and measures the (birational) variation of the fibres.
All the theory about the canonical bundle formula is developed for lc-trivial fibrations.
The definition of this class of fibrations is quite technical and for it we refer to the second section.
It is shown in ~\cite{Amb04} by Ambro, for $(X,B)$ generically klt on the base,
and in ~\cite{Corti} by Koll\'ar in the lc case
the following result
\begin{teo}[Ambro, ~\cite{Amb04} Theorem 0.2,  Koll\'ar, ~\cite{Corti}]
Let $f \colon (X,B) \rightarrow Z$ be an lc-trivial fibration.
Then there exists a proper birational morphism
$Z' \rightarrow Z$ with the following properties:
\begin{enumerate}
\item $K_{Z'}+B_{Z'}$ is a $\mathbb{Q}$-Cartier divisor, 
and $\nu^{\ast}(K_{Z'}+B_{Z'}) = K_{Z''}+B_{Z''}$ 
for every proper birational morphism $\nu\colon Z' \rightarrow Z''$.
\item the divisor $M_{Z'}$ is $\mathbb{Q}$-Cartier and nef
and $\nu^{\ast}(M_{Z'}) = M_{Z''}$ for every
proper birational morphism $\nu\colon Z' \rightarrow Z''$.
\end{enumerate}
\end{teo}

The regularity of the pair $(Z,B_Z)$ depends on the
regularity of $(X,B)$, more precisely $(Z,B_Z)$ is klt (resp. lc)
if and only if $(X,B)$ is (see ~\cite[Proposition 3.4]{Amb99}).

Furthermore the following properties are conjectured for $M_Z$.
\begin{conjecture}[Prokhorov-Shokurov, ~\cite{PS} Conjecture 7.13]\label{bigconj}
Let $f\colon(X,B)\rightarrow Z$ be an lc-trivial fibration.
\begin{enumerate}
\item (Log Canonical Adjunction) There exists a proper birational morphism $Z'\rightarrow Z$
such that $M_{Z'}$ is semiample.
\item (Particular Case of Effective Log Abundance Conjecture) Let
$X_{\eta}$ be the generic fibre of $f$. Then $I_0 (K_{X_{\eta}} +B_{\eta})\sim 0$, where $I_0$ depends
only on $\dim X_{\eta}$ and the multiplicities of the horizontal part of $B$.
\item (Effective Adjunction) The divisor $M_Z$ is effectively semiample, that
is, there exists a positive integer $I$ depending only on the dimension
of $X$ and the horizontal multiplicities of $B$ (a finite set of rational
numbers) such that $IM_Z$ is the pullback of $M$,
where $M$ is a base point free divisor on some model $Z'/Z$.
\end{enumerate}
\end{conjecture}

The relevance of the above conjecture is well illustrated for instance
by a remark due to X. Jiang, who observed recently ~\cite[Remark 7.3]{J2}
that Conjecture \ref{bigconj}(3) implies a uniformity statement
for the Iitaka fibration of \textit{any} variety of positive Iitaka
dimension under the assumption that the fibres have a good minimal model.

These conjectures are proved in the case where the fibres have dimension one.
\begin{teo}[Prokhorov-Shokurov, ~\cite{PS}]\label{PSteo}
Conjecture \ref{bigconj} holds in the case $\dim X = \dim Z + 1$.
\end{teo}
It is important to remark that the proof of Theorem \ref{PSteo} strongly uses the existence
of the moduli space $\mathcal{M}_{0,n}$.
Moreover the constant $I$ that appears in  Theorem \ref{PSteo} is not explicitely determined.
In ~\cite[Remark 8.2]{PS} the authors expect that a sharp result might be $I=12r$ where $r$ is as in Formula (\ref{CBFamb}).
In particular this would imply that the denominators of the $\mathbb{Q}$-divisor $M$ are bounded by $r$.
In the case of one-dimensional fibre, if $B=0$
the general fibre is an elliptic curve and the result follows from Kodaira's Formula (\ref{kodsup}).
If $B\neq 0$ then the generic fibre $F$ is a rational curve and $B$
is effective and such that ${\rm deg}B|_F=2$.
In this case the situation is more complicated.

In this work we prove that in the case where the generic fibre is a rational curve
the expectation of Prokhorov and Shokurov cannot be true.
Indeed we can prove that there are examples in which $12rM$
has not even integer coefficients.
\begin{cex}
There exists an lc-trivial fibration $f\colon(X,B)\rightarrow Z$ 
whose generic fibre is a rational curve
such that $12rM_Z$ has not integer coefficients.
More precisely for any positive and odd $r\in \mathbb{N}$
there exists an lc-trivial fibration $f\colon(X,B)\rightarrow Z$
such that (\ref{CBFamb}) holds and
with moduli divisor $M_Z=\sum c_p p$
and there exists a point $o\in Z$ such that
the minimal integer $m$ such that $m c_o\in\mathbb{Z}$
is greater or equal to $2r^2-r$.
\end{cex}
Neverthless we can show the following local result, which is not far from being sharp by the previous example:
\begin{teo}\label{mainprop}
Let $f\colon (X,B)\rightarrow Z$ be an lc-trivial fibration
whose generic fibre is a rational curve.
Let $B_Z=\sum \beta_i p_i$ be the discriminant.
Then for every $i$ there exists $l_i\leq 2r$
such that $rl_i\beta_i\in \mathbb{Z}$.
\end{teo}
An important remark is that for an lc-trivial fibration whose general fibre is a rational curve,
for every $I\in\mathbb{Z}$,
$Ir M_Z$ has integer coefficients if and only if $Ir B_Z$ has integer coefficients.
To prove Theorem \ref{mainprop} we give an expression of the log canonical threshold
of a fibre with respect to $(X,B)$ 
in terms of the pull back of the canonical divisor of $X$,
the pull back of the fibre and the pull back of $B$.

An interesting question is to determine the best possible \textit{global}
bound on the denominators of $M_Z$.
Theorem \ref{mainprop} implies that
$(2r)!M_Z$ has integer coefficients, but it is certainly not the best bound.
Using techniques from Theorem \ref{mainprop}
we can prove that a polynomial global bound cannot exist and determine a bound.
\begin{teo}\label{mainteo}
\begin{enumerate}
	\item A polynomial global bound on the denominators of $M_Z$ cannot exist.
	Precisely for all $N$ there exists an lc-trivial fibration
	$$f\colon(X,B)\rightarrow Z$$
	such that if $V$ is the smallest integer such that $VM_Z$ has integer coefficients
	then $$V\geq r^{N+1}.$$
	\item Let $f\colon (X,B)\rightarrow Z$ be an lc-trivial fibration
whose generic fibre is a rational curve.
Then there exists an integer $N(r)$ that depends only on $r$ such that
$N(r)M_Z$ has integer coefficients.
More precisely if we set
$s(q)=\max\{s \mid q^s\leq 2r\}$ then
$$N(r) = r\prod_{\begin{array}{c}
\scriptstyle q \leq 2r \\
\scriptstyle q\; {\rm prime}
\end{array} }
q^{s(q)}.$$
\item For all $r$ odd there exists an lc-trivial fibration
$$f\colon(X,B)\rightarrow Z$$
	such that if $V$ is the smallest integer such that $VM_Z$ has integer coefficients
	then $V=N(r)/r$.
\end{enumerate}
\end{teo}

In ~\cite{Tod} G. T. Todorov proves, 
in the case where the pair $(X,B)$ is klt over the generic point of $Z$,
the existence of an explicitely computable integer $I(r)$ such that $I(r)M_Z$
has integer coefficients using techniques from ~\cite{FM} where the existence of such an integer is proved in the case $B=0$.
Todorov's bound is considerably greater than 
the bound provided by Theorem \ref{mainteo}:
\hspace{1cm}

\begin{center}
\begin{tabular}{l|l|l}
r & I(r) & N(r) \\
\hline
3 & 120 & 60 \\
4 & 5040 & 420 \\
5 & 1441440 & 2520 \\
6 & 160626866400 & 27720 \\
7 & 288807105787200 & 360360 \\
8 & 6198089008491993412800 & 360360 \\
9 & 7093601304616933605068169600 & 12252240 \\
10 & 194603155528763897469736633833782400 & 232792560 \\
\end{tabular}
\end{center}

An explicit global bound on the denominators of $M_Z$
is important in order to obtain effective results
for the pluri-log-canonical maps
of pairs with positive Kodaira dimension.
For instance the bounds in ~\cite[Theorem 6.1]{FM} and ~\cite[Theorem 4.2]{Tod}
can be immediately improved by using Theorem \ref{mainteo}.\\
One of the difficulties of studying the moduli part of lc-trivial fibrations
with fibres of dimension greater than one is the lack of a moduli space for the fibres.
It is therefore worth noticing that
our arguments make no use of $\mathcal{M}_{0,n}$.
We hope that our more elementary approach could lead to a better understanding
of the moduli divisor for fibrations with higher dimensional fibres.

\bigskip

\thanks{{\bfseries Acknowledgements.}
This work has been written during my first year of Ph.D at the Universit\'e de Strasbourg.
I would like to thank my advisor Gianluca Pacienza for proposing me to work on this subject and for constantly supporting me with his invaluable remarks and his advice.
I would also like to thank Florin Ambro for the many fruitful conversations we have had together as well 
as Andreas H\"oring for all the useful comments he has made to me.
}

\section{Notations and preliminaries}
\subsection{Notations, definitions and known results}
We will work over $\mathbb{C}$,
In the following $\equiv$, $\sim$ and $\sim_{\mathbb{Q}}$ will respectively indicate
numerical, linear and $\mathbb{Q}$-linear equivalence of divisors.
The following definitions are taken from ~\cite{KM}.

\begin{defi}
Let $(X,B)$ be a pair, $B=\sum b_i B_i$ with $b_i \in\mathbb{Q}$. 
Suppose that $K_X+B$ is $\mathbb{Q}$-Cartier.
Let $\nu\colon Y\rightarrow X$ be a birational morphism, $Y$ normal.
We can write
$$K_Y\equiv \nu^{\ast}(K_X+B)+\sum a(E_i,X,B) E_i.$$
where $E_i\subseteq Y$ are distinct prime divisors and $a(E_i,X,B)\in\mathbb{R}$.
Furthermore we adopt the convention that a nonexceptional divisor $E$ appears in the sum
if and only if $E=\nu_{\ast}^{-1}B_i$ for some $i$
and then with coefficient $a(E,X,B)=-b_i$.\\
The $a(E_i,X,B)$ are called discrepancies.
\end{defi}

\begin{defi}
Let $(X,B)$ be a pair and $f\colon X\rightarrow Z$ be a morphism.
Let $o\in Z$ be a point (possibly of positive dimension).
A log resolution of $(X,B)$ over $o$
is a birational morphism $\nu\colon X'\rightarrow X$
such that for all $x\in f^{-1}o$
the divisor $\nu^{\ast}(K_X+B)$ is
simple normal crossing at $x$.
\end{defi}

\begin{defi}
We set
$$\rm{discrep}(X,B)
=\inf \{a(E,X,B)\;|\; E\, exceptional\, divisor\, over\, X\}.$$
A pair $(X,B)$ is defined to be
\begin{itemize}
\item klt (kawamata log terminal) if $\rm{discrep}(X,B)>-1$,
\item lc (log canonical) if $\rm{discrep}(X,B) \geq -1.$
\end{itemize}
\end{defi}

\begin{defi}
Let $f\colon (X,B)\rightarrow Z$ be a morphism and $o\in Z$ a point.
For an exceptional divisor $E$ over $X$ we set
$c(E)$ its image in $X$.
We set $$\rm{discrep}_o(X,B)
=\inf \{a(E,X,B)\;|\; E\, {\rm exceptional}\, {\rm divisor}\, {\rm over}\, X,\; f(c(E))=o\}.$$
A pair $(X,B)$ is defined to be
\begin{itemize}
\item klt over $o$ (kawamata log terminal) if $\rm{discrep}_o(X,B)>-1$,
\item lc over $o$ (log canonical) if $\rm{discrep}_o(X,B) \geq -1.$
\end{itemize}
\end{defi}

\begin{defi}
Let $(X,B)$ be an lc pair, $D$ an effective $\mathbb{Q}$-Cartier $\mathbb{Q}$-divisor.
The log canonical threshold of $D$ for $(X,B)$ is
$$\gamma=\sup\{t\in\mathbb{R}^+|\,(X,B+tD)\: is\: lc\}.$$
\end{defi}

\begin{defi}
Let $(X,B)$ be a lc pair, $\nu\colon X'\rightarrow X$ a log resolution.
Let $E\subseteq X'$ be a divisor on $X'$ of discrepancy $-1$. Such a divisor is called a log canonical place. 
The image $\nu(E)$ is called center of log canonicity of the pair.
If we write $$K_{X'}\equiv \nu^{\ast}(K_X+B)+E,$$
we can equivalently define a place as an irreducible component of $\lfloor -E \rfloor$.
\end{defi}

\begin{defi}
Let $(X,B)$ be a pair and $\nu\colon X'\rightarrow X$
a log resolution of the pair.
We set $$A(X,B)=K_{X'}-\nu^{\ast}(K_X+B)$$ and
$$A(X,B)^{\ast}=A(X,B)+\sum_{E\,{\rm place}}E.$$
\end{defi}

\begin{defi}
A lc-trivial fibration $f \colon (X,B) \rightarrow Z$ consists of a contraction
of normal varieties $f \colon X \rightarrow Z$ and of a log pair $(X,B)$ satisfying
the following properties:
\begin{enumerate}
\item $(X,B)$ has log canonical singularities over a big open subset $U\subseteq Z$;
\item ${\rm rank}\, f'_{\ast}\,\mathcal{O}_X(\lceil A^{\ast}(X,B)\rceil) = 1$
where $f'=f\circ\nu$ and $\nu$ is a given log resolution of the pair $(X,B)$;
\item there exists a positive integer $r$, a rational function $\varphi\in k(X)$
and a $\mathbb{Q}$-Cartier divisor $D$ on $Z$ such that
$$K_X + B +\frac{1}{r}(\varphi) = f^{\ast}D.$$
\end{enumerate}
\end{defi}

\begin{oss}\label{cartindex}
{\rm
The smallest possible $r$ is the minimum of the set
$$\{m\in\mathbb{N}|\,m(K_X+B)|_F\sim 0\}$$
that is the Cartier index of the fibre.
We will always assume that the $r$ that appears in the formula is the smallest.
}
\end{oss}

\begin{defi}
Let $p\subseteq Z$ be a codimension one point.
The log canonical threshold of $f^{\ast}(p)$ with respect to the pair $(X,B)$ is
$$\gamma_p=\sup\{t\in\mathbb{R}|\,(X,B+tf^{\ast}(p)) {\rm \:is\: lc\: over\:} p\}.$$
We define the {\rm discriminant} of $f \colon (X,B) \rightarrow Z$ as
\begin{eqnarray}\label{discriminant}
B_Z&=&\sum_{p}(1-\gamma_p)p.
\end{eqnarray}
\end{defi}
We remark that, since the above sum is finite, $B_Z$ is a $\mathbb{Q}$-Weil divisor.
\begin{oss}
{\rm
In what follows we will treat the case where $f\colon X\rightarrow Z$
is a $\mathbb{P}^1$-bundle over a smooth curve.
We write $B$ as the sum of its vertical part and its horizontal part,
$B=B^h+B^v$.
Since every fibre of $f$ is irreducible there exists a $\mathbb{Q}$-divisor $\Delta$ on $Z$
such that $B^v=f^{\ast}\Delta$.
This implies that also $f\colon(X,B^h)\rightarrow Z$
is an lc-trivial fibration
and let $B'_Z$ and $M'_Z$ be its discriminant and moduli part.
Then by ~\cite[Remark 3.3]{Amb04}
$B_Z=B'_Z+\Delta$ and $M_Z=M'_Z$.
Thus we can suppose $B=B^h$.
In this case,
if we write $B=\sum b_i B_i$, the smallest possible $r$
is the least common multiple of the denominators of the $b_i$'s
and for all $i$ $$b_i\in\frac{1}{r}\mathbb{Z}.$$
}
\end{oss}
\begin{oss}
{\rm
Let $f\colon (X,B)\rightarrow Z$ be an lc-trivial fibration on a smooth curve
and let $o\in Z$ be a point.
Let $F=f^{\ast}o$ be its fibre.
Let $\delta\colon \hat{X}\rightarrow X$ be a log resolution of $(X, B+f^{\ast}o)$
over $o$, that is, if $E$ is an exceptional curve of $\delta$
then $f(\delta(E))=o$.
Then we have
$$
\begin{array}{rcl}
\delta^{\ast}K_X&=&K_{\hat{X}}-\sum e_i E_i\\
\delta^{\ast}F&=&\tilde{F}+\sum a_i E_i\\
\delta^{\ast}B&=&\tilde{B}+\sum \alpha_i E_i
\end{array}
$$
The resolution $\delta$ is a log-resolution over $o$ also for the pair $(X, B+tF)$
for all $t$.
If $(X,B+tF)$ is lc then by definition for all $i$
$$-e_i+ta_i+\alpha_i\leq 1.$$
Since the coefficient of $F$ has to be less or equal than one, we also have $t\leq 1$.
Therefore
$$t\leq\min\{1,\min_i\{\frac{1}{a_i}(1+e_i-\alpha_i)\}\}.$$ 
}
\end{oss}

\begin{defi}
Fix $\varphi\in\mathbb{C}(X)$ such that $K_X + B +\frac{1}{r}(\varphi) = f^{\ast}D$.
Then there exists a unique divisor $M_Z$ such that we have
\begin{eqnarray}\label{cbf}
K_X + B +\frac{1}{r}(\varphi) &=& f^{\ast}(K_Z+B_Z+M_Z)
\end{eqnarray}
where $B_Z$ is as in (\ref{discriminant}).
The $\mathbb{Q}$-Weil divisor $M_Z$ is called the {\rm moduli part}.
\end{defi}
We have the two following results.\\
\begin{teo}~\cite[Theorem 2.5]{Amb04}, ~\cite{Corti}\label{nefness}
Let $f \colon(X,B)\rightarrow Z$ be a lc-trivial fibration.
Then there exists a proper birational morphism
$Z'\rightarrow Z$ with the following properties:
\begin{description}
\item[(i)] $K_{Z'}+B_{Z'}$ is a $\mathbb{Q}$-Cartier divisor, 
and $\nu^{\ast}(K_{Z'}+B_{Z'}) = K_{Z''}+B_{Z''}$
for every proper birational morphism $\nu \colon Z''\rightarrow Z'$.
\item[(ii)] $M_{Z'}$ is a nef $\mathbb{Q}$-Cartier divisor and $\nu^{\ast}(M_{Z'}) = M_{Z''}$ for every
proper birational morphism $\nu \colon Z''\rightarrow Z'$.
\end{description}
\end{teo}
\begin{teo}[Inverse of adjunction] ~\cite[Proposition 3.4]{Amb99}\label{invadj} Let $f\colon(X,B)\rightarrow Z$ be a lc-trivial
fibration. 
Then $(Z,B_Z )$ has klt (lc)
singularities in a neighborhood of a point $p\in Z$ if and only if $(X,B)$ has
klt (lc) singularities in a neighborhood of $f^{-1}p$.
\end{teo}

The Formula (\ref{cbf}), with the properties stated in Theorem \ref{nefness} and Theorem \ref{invadj} is called \textit{canonical bundle formula}.

\subsection{A useful result on blow-ups on surfaces}
Let $X$ be a smooth surface.
Let $\delta\colon\hat{X}\rightarrow X$ be a sequence of blow-ups,
$\delta=\varepsilon_h\circ\ldots\circ \varepsilon_1$ and denote $p_i$ the point blown-up by $\varepsilon_i$.
In what follows by abuse of notation we will denote with $E_i$
the exceptional curve of $\varepsilon_i$
as well as its birational transform in further blow-ups.
\textbf{In what follows we will suppose that in ${\rm Exc}(\delta)$ there is just one $(-1)$-curve.}
Since the exceptional curve $E_h$ of $\varepsilon_h$ is a $(-1)$-curve
it is the only exceptional curve of ${\rm Exc}(\delta)$.
Suppose that the first point $p_1$ that is blown-up
belongs to a smooth curve $F$.
We will denote by $\tilde{F}$ the strict transform of $F$
by $\varepsilon_i\circ\ldots\circ \varepsilon_1$ for all $i$.

\begin{lem}\label{3ind}
Let $f\colon (X,B)\rightarrow Z$ be a $\mathbb{P}^1$-bundle on a smooth curve $Z$
and suppose that $B=(2/d)D$ 
where $D$ is a reduced divisor such that $DF=d$.
Suppose moreover that there is a point $o\in Z$
such that $D$ is tangent to $F=f^{\ast}o$ at a smooth point of $D$
with multiplicity $d/2\leq l<d$.
Then the log canonical threshold
$$\gamma:=\gamma_o=\sup\{t\in\mathbb{R}|\,((X,B),tf^{\ast}o)\; {\rm is} \; {\rm lc} \; {\rm  over}\; o\}$$
has the following expression
$$\gamma = 1+\frac{1}{l}-\frac{2}{d}.$$
\end{lem}
\begin{proof}
A log resolution for the pair $(X,2/dD+\gamma_o F)$
over $o$ is a sequence of blow-ups $\delta=\varepsilon_l\circ\ldots\circ\varepsilon_1$
such that a picture of the $(l-1)$-th step is\\

\vspace{2cm}
\setlength{\unitlength}{1cm}
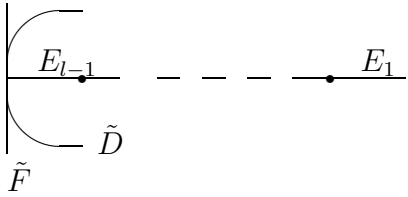
\begin{picture}(1,1)
\put(1,0){\line(0,1){2}}
\put(1,-0.5){$\tilde{F}$}
\put(2.2,0){$\tilde{D}$}
\put(1.4,1.1){$E_{l-1}$}
\put(2,1){\oval(2,1.8)[l]}
\put(1,1){\line(1,0){1.5}}
\put(2,1){\circle*{0.1}}
\put(3,1){\line(1,0){0.3}}
\put(3.6,1){\line(1,0){0.3}}
\put(4.2,1){\line(1,0){0.3}}
\put(4.8,1){\line(1,0){1.5}}
\put(5.3,1){\circle*{0.1}}
\put(5.7,1.1){$E_1$}
\linethickness{1mm}
\end{picture}
\vspace{1cm}

Then $$\delta^{\ast}D=\tilde{D}+\sum_{j=1}^l jE_j$$
and we have
$$\delta^{\ast}(\frac{2}{d}D)=\frac{2}{d}\tilde{D}+\frac{2}{d}\sum_{j=1}^l j E_j. $$
By definition $\alpha_l$ is the coefficient of $\delta^{\ast}(2/dD)$
at $E_l$, and by our computation it is $2l/d$.
Since
\begin{eqnarray*}
\gamma&=&\min\{1,\min_{i=1\ldots l}\{1+\frac{1}{i}-\frac{2}{d}\}\}\\
&=&\min\{1,1+\frac{1}{l}-\frac{2}{d}\}
\end{eqnarray*}
we obtain
$$\gamma=1+\frac{1}{l}-\frac{2}{d}.$$
\end{proof}

\section{Local results}

In this section we will be always in the situation where the fibres have dimension 1.
In this case, if $B=0$ the condition that $K_F$ is torsion implies
the generic fibre is an elliptic curve.
If $B\neq 0$ then $F$ has to be a rational curve
and the second condition in the definition of the lc-trivial fibration
implies that the horizontal part of $B$ is effective.

Thanks to the following lemma, studying the denominators of $M_Z$
is the same thing as studying the denominators of $B_Z$.
\begin{lem}
Let $f\colon (X,B)\rightarrow Z$ be an lc-trivial fibration whose general fibre is a rational curve.
Then for all $I\in \mathbb{N}$ $IrB_Z$ has integer coefficients if and only if $IrM_Z$
has integer coefficients.
\end{lem}
\begin{proof}
By cutting with sufficiently general hyperplane sections we can assume that $\dim Z=1$.\\
We write the canonical bundle formula for $f\colon (X,B)\rightarrow Z$:
$$K_X+B+\frac{1}{r}(\varphi)=f^{\ast}(K_Z+B_Z+M_Z).$$
Let $\nu\colon \hat{X}\rightarrow X$ be a desingularization of $X$, let $\hat{B}$
be the divisor defined by $$K_{\hat{X}}+\hat{B}=\nu^{\ast}(K_X + B)$$ and $\hat{f}=f\circ\nu$.
Then $\hat{f}\colon(\hat{X},\hat{B})\rightarrow Z$ is lc-trivial and has the same discriminant as $f$.
Moreover it has the same moduli divisor, since
$$K_{\hat{X}}+\hat{B}+\frac{1}{r}(\varphi)=\nu^{\ast}(K_X+B)+\frac{1}{r}(\varphi)=\hat{f}^{\ast}(K_Z+B_Z+M_Z).$$
The surface $\hat{X}$ is smooth and $\hat{X}\rightarrow Z$ has generic fibre $\mathbb{P}^1$
then there exists a birational morphism defined over $Z$
$$
\xymatrix{
\hat{X} \ar[d]_{\hat{f}} \ar[r]&X' \ar[ld]_{f'}\\
Z& }
$$ 
where $f'\colon X'\rightarrow Z$ is a $\mathbb{P}^1$-fibration.
It follows that each fibre of $\hat{f}$ has an irreducible component with coeffient one.
Then the statement follows from the equality
$$r(K_{\hat{X}}+\hat{B})+(\varphi)=r\hat{f}^{\ast}(K_Z+B_Z+M_Z).$$

\end{proof}

\begin{teo}\label{dim2}
Let $f\colon X\rightarrow Z$ be a $\mathbb{P}^1$-bundle with $\dim X=2$.
Let $o\in Z$ be a point and 
$\gamma$ be the log canonical threshold of $f^{\ast}o$ with respect to $(X,B)$.
Then there is a constant
$m\leq 2r^2$ such that $m \gamma$ is integer.
Such an $m$ is of the form $lr$ where $l\leq 2r$.
\end{teo}
\begin{proof}
The pair $(X,B+\gamma F)$ is lc and not klt, that is, it has an lc centre.
There are now two cases.\\
{\bfseries The centre has dimension one.}\\
If the centre has dimension one, then it is the whole fibre because all the fibres are irreducible.
In this case we have
$$1={\rm mult}_F (B+\gamma F)={\rm mult}_F (B)+\gamma$$
and since $r{\rm mult}_F (B)\in\mathbb{Z}$ also $r\gamma\in\mathbb{Z}$.\\
{\bfseries The centre has dimension zero.}\\
\textsl{Step 1}
Take $\nu\colon X'\rightarrow X$
a log resolution of $(X,B+\gamma F)$. Notice that the fibre over $o$ is a tree of $\mathbb{P}^1$'s.\\
Since $(X,B+\gamma F)$ is lc and not klt there is a place appearing between the leaves of the tree.
Write $\nu$ as a composition of blow-ups, set $\nu=\varepsilon_N\circ\ldots\circ\varepsilon_1$
and let $k$ be the minimum of the indices such that
the exceptional curve of $\varepsilon_k$ is a place for $(X,B+\gamma F)$, $P=E_k$.
Let $\eta$ be the composition $\varepsilon_k\circ\ldots\circ\varepsilon_1\colon X_1\rightarrow X$.
We have:
$$
\xymatrix{
X' \ar[dd]_{\nu} \ar[rd]\\
&X_1\ar[ld]_{\eta}\\
X & 
& }
$$
If the only $(-1)$-curve in $X_1$ is $P$
then we set $\hat{X}=X_1$ and $\delta:=\eta$.
Otherwise, if there is another $(-1)$-curve,
by the Castelnuovo's theorem we can contract it in a smooth way:

$$
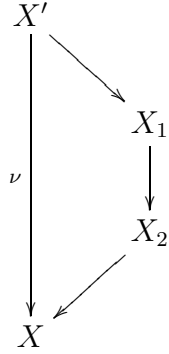
\xymatrix{
X' \ar[ddd]_{\nu} \ar[rd]\\
&X_1\ar[d]\\
&X_2\ar[ld]\\
X & 
& }
$$
This process ends because in $X'$ there were finitely many $\nu$-exceptional curves.
Then we obtain a smooth surface $\hat{X}$ such that the only $(-1)$-curve
in $X$ is $P$.
We set $\delta\colon\hat{X}\rightarrow X$ and write $\delta=\varepsilon_h\circ\ldots\circ\varepsilon_1$.\\
\textsl{Step 2}
We have obtained $\hat{X}$ smooth with a diagram
$$
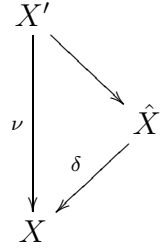
\xymatrix{
X' \ar[dd]_{\nu} \ar[rd]\\
&\hat{X}\ar[ld]_{\delta}\\
X & 
& }
$$
where $\hat{X}\rightarrow X$ is minimal in order to obtain
a log canonical place $P$ which has to be a $-1$-curve
and $\delta=\varepsilon_h\circ\ldots\circ\varepsilon_1$ is a sequence of blow ups.
Let $p_i$ be the point blown up by $\varepsilon_i$.
Let $\tilde{B}^j_i$ be the strict transform
of the component $B_i$ of $B$ at the step $j$
and $\tilde{B}^j$ be the strict transform of $B$. 
By abuse of notation we will denote by $\tilde{F}$
the strict transform of $F$ by every $\varepsilon_i$
and by $E_i$ the exceptional curve 
of $\varepsilon_i$ as well as its strict transform in the further blow-ups.
Notice that $P=E_h$.
In what follows we will adopt the following notation:
$$B=\sum b_i B_i;$$
$$\delta^{\ast}K_X=K_{\hat{X}}-\sum e_i E_i;\;\;\;\;
\delta^{\ast}B=\tilde{B}+\sum \alpha_i E_i;\;\;\;\; 
\delta^{\ast}F=\tilde{F}+\sum a_i E_i.$$
Here $\tilde{B}$ and $\tilde{F}$ denote the strict transform of $B$ and $F$.
Remark that for all $i$ we have 
\begin{eqnarray}\label{diesis}
\alpha_i&\in& \frac{1}{r}\mathbb{Z}.
\end{eqnarray}
Indeed $b_i\in 1/r\mathbb{Z}$ for all $i$ by Remark \ref{cartindex}. Equation (\ref{diesis})
follows from the fact that
$$\alpha_1=\sum_{B_i\ni p_1}b_i {\rm mult}_{p_1}B_i$$
and, for $l>1$, that $\alpha_l$ is a linear combination of the $\alpha_j$'s
with $j<l$
plus $\sum_{\tilde{B}^{l-1}_i\ni p_l}b_i {\rm mult}_{p_l}\tilde{B}^{l-1}_i$.

Since $E_h$ is a place we have
$$1={\rm mult}_{E_h}(\delta^{\ast}(K_X+B+\gamma F)-K_{\hat{X}})=-e_h+\alpha_h+\gamma a_h.$$
Since $e_h$ is an integer and $\alpha_h\in 1/r\mathbb{Z}$,
if we prove that $a_h\leq 2r$ we are done.
By the minimality of $\delta$ there exists a component $B_1$ of $B$
such that the strict transform $\tilde{B}_1^h$ of $B_1$ meets $E_h$,
that is $\tilde{B}_1^h E_h>0$.
Then
\begin{eqnarray*}
2r&\geq&B_1 F=\delta^{\ast}B_1 \delta^{\ast}F=\tilde{B}_1^h\delta^{\ast}F=\tilde{B}_1^h(\tilde{F}+\sum a_i E_i)\\
&\geq&a_h\tilde{B}_1^h E_h\geq a_h.
\end{eqnarray*}

\end{proof}

We can finally prove the main result.

\begin{proof} [Proof of Theorem \ref{mainprop}]
The statement in dimension 2 follows from Theorem \ref{dim2}
and ~\cite[Lemma 2.6]{Amb04}.
Indeed if $X\rightarrow Z$ is a fibration whose general fibre is a $\mathbb{P}^1$
and $X$ is smooth, then by the general theory of smooth surfaces 
there exists a birational morphism $\sigma\colon X\rightarrow X'$ where $X'$ is a 
$\mathbb{P}^1$-bundle. 
More precisely $X'$ is a minimal model of $X$ that is unique if
the genus of $Z$ is positive.\\
The general result follows from the one in dimension 2 by induction
on the dimension of the base.
Suppose now that the statement is true in dimension $n-1$
and let $X\rightarrow Z$ be a fibration of dimension $n$.
The set $$\mathcal{S}=\left\{
\begin{array}{l}
\scriptstyle 
o\; {\rm point\; of}\; Z\;{\rm of\; codimension\; 1\; such\; that\;
the\; log\; canonical} \\
\scriptstyle {\rm threshold\; of}\; f^{\ast}o\;
{\rm with\; respect\; to}\; (X,B)\; {\rm is\; different\; from}\; 1
\end{array} 
\right\}$$
is a finite set.\\
We fix then a point $o\in\mathcal{S}$.
By the Bertini theorem,
since $Z$ is smooth,
we can find a hyperplane section $H\subseteq Z$ such that
\begin{enumerate}
\item $H$ is smooth;
\item $H$ intersects $o$ transversally;
\item $H$ does not contain any intersection $o\cap o'$
where $o'\in\mathcal{S}\backslash\{o\}$.
\end{enumerate}
Set $$X_H=f^{-1}(H); \;\;f_H=f|_{X_H}; \;\;B_H=B|_{X_H};\;\; o_H=o\cap H.$$
The restriction $f_H\colon (X_H,B_H)\rightarrow H$
is again an lc-trivial fibration.
Then the log canonical threshold of $f_H^{\ast}o_H$ 
with respect to $(X_H,B_H)$ is equal to
the log canonical threshold of $f^{\ast}o$ 
with respect to $(X,B)$
and the theorem follows from the inductive hypothesis.
\end{proof}

Notice that even if in many cases $m=r$ is sufficient to have that $mM_Z$ has integer coefficients
there exist cases
in which a greater coefficient is needed.
\begin{ex}\label{cex}
{\rm
Let $\pi\colon X\rightarrow C$ be a $\mathbb{P}^1$-bundle on a curve $C$.
Let $X^0\rightarrow U$ be a local trivialization, where $U\subseteq C$
is an open subset and $X^0=\pi^{-1}U$. This means that there is a commutative diagram
$$
\xymatrix{
X^0 \ar[d]_{\pi} \ar[r]^{\sim}
&U\times\mathbb{P}^1\ar[ld]^{p_1}\\
U. }
$$
We can furthermore suppose that we have a local coordinate $t$ on $U$.
Let $[x:y]$ be coordinates on $\mathbb{P}^1$.
Set $$D=\{t y^d-x^{l}y^{d-l}-x^d=0\}\subseteq U\times\mathbb{P}^1$$
and let $\bar{D}$ be the Zariski closure of $D$ in $X$.\\
Consider the pair $(X,2/d\bar{D})$.
Then we have ${\rm deg}(K_X+2/d\bar{D})|_F=0$ and there exists
a rational function $\varphi$ such that we can write 
$$K_X+2/d\bar{D}+\frac{1}{r}(\varphi)=f^{\ast}(K_C+B_C+M_C)$$
where $r=d$ if $d$ is odd and $r=d/2$ if $d$ is even.
We want to compute now the coefficient of the divisor $B_C$ at the point $t=0$.
Its coefficient is $1-\gamma$
where $\gamma$ is the log canonical threshold of $((X,2/d\bar{D}),F)$.
A log resolution for the pair $(X,2/d\bar{D})$ over the point $t=0$
is given by the composition of $l$ blow-ups.
At the $(l-1)$-th step the picture is as follows\\

\vspace{2cm}
\setlength{\unitlength}{1cm}
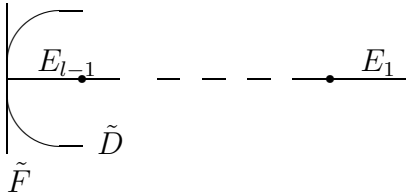
\begin{picture}(1,1)
\put(1,0){\line(0,1){2}}
\put(1,-0.5){$\tilde{F}$}
\put(2.2,0){$\tilde{D}$}
\put(1.4,1.1){$E_{l-1}$}
\put(2,1){\oval(2,1.8)[l]}
\put(1,1){\line(1,0){1.5}}
\put(2,1){\circle*{0.1}}
\put(3,1){\line(1,0){0.3}}
\put(3.6,1){\line(1,0){0.3}}
\put(4.2,1){\line(1,0){0.3}}
\put(4.8,1){\line(1,0){1.5}}
\put(5.3,1){\circle*{0.1}}
\put(5.7,1.1){$E_1$}
\linethickness{1mm}
\end{picture}
\vspace{1cm}

We call $\delta\colon \hat{X}\rightarrow X$ this composition of blow-ups. We have
$$\delta^{\ast}K_X=K_{\hat{X}}-\sum_{i=1}^l iE_i\;\;\;\;
\delta^{\ast}\bar{D}=\tilde{D}+\sum_{i=1}^liE_i\;\;\;\;
\delta^{\ast}F=\tilde{F}+\sum_{i=1}^l iE_i,$$
where by abuse of notation we denote by $E_i$ the exceptional divisor of the $i$-th blow-up
as well as its strict transforms after the following blow-ups. 
Thus 
$$\delta^{\ast}(K_X+2/d\bar{D}+\gamma F)=K_{\hat{X}}+2/d\tilde{D}+\gamma\tilde{F}+\sum_{i=1}^l
i(-1+\gamma+2/d)E_i.$$
By Lemma \ref{3ind} we have $$\gamma=1+\frac{1}{l}-\frac{2}{d}.$$
So if we chose $l<d$ and such that $2l>d$, we obtain $\gamma=1-\frac{2l-d}{ld}$.
For $l=5$ and $d=9$ we have 
$\gamma=1-\frac{1}{45}\notin \frac{1}{12r}\mathbb{Z}$
contrary to the Prokhorov and Shokurov expectation.\\
Notice that this gives us an example also if we take $l$ to be any prime
greater or equal to 13 and $d=2l-1$.

To prove that the bound stated in Theorem \ref{dim2} is not far from being sharp,
we take $d$ even such that $d/2$ is odd and $l=d-1$. Then $r=d/2$ and
$$\gamma=1-\frac{2l-d}{ld}=1-\frac{2(2r-1)-2r}{2r^2-r}=1-\frac{2(2r-1)-2r}{2r^2-r}=1-\frac{2(r-1)}{(2r-1)r}.$$
Since $2(r-1)$ and $(2r-1)r$ are coprime,
the smallest integer $m$ such that $m\gamma$ is integer is $m=2r^2-r$.
}
\end{ex}

\section{Global results}
\begin{lem}\label{spiegone}
Let $f\colon X\rightarrow Z$ be a $\mathbb{P}^1$-bundle on a smooth curve $Z$.
Let $D\subseteq X$ be a reduced divisor such that
$f|_D\colon D\rightarrow Z$ is a ramified covering of degree $d$ with at least $N$
ramification points $p_1\ldots p_N$ that are smooth points for $D$.
Suppose that $d$ is even.
Suppose moreover that the ramification indices $l_1,\ldots, l_N$ at $p_1,\ldots,p_N$
satisfy the following properties:
\begin{enumerate}
	\item $2l_i\geq d$ for all $i$;
	\item $l_i$ and $l_j$ are coprime for all $i\neq j$;
	\item $l_i$ and $d$ are coprime for all $i$.
\end{enumerate}
Then
\begin{description}
\item[(i)] the fibration $$f\colon (X,2/dD)\rightarrow Z$$
is an lc-trivial fibration, in particular there exists a rational function $\varphi$
such that
$$K_X+\frac{2}{d}D+\frac{1}{r}(\varphi)=f^{\ast}(K_Z+M_Z+B_Z).$$
\item[(ii)] The Cartier index of the fibre is $r=d/2$.
\item[(iii)] Let $V$ be the smallest integer such that $VM_Z$ has integer coefficients.\\
Then $V\geq r^{N+1}$.
\end{description}
\end{lem}
\begin{proof}
The first part of the statement follows easily from the fact the degree of $(K_X+2/dD)|_F$
is $0$.
The Cartier index of the fibre is
$$r=\min\{m|\,m(K_X+2/dD)|_F\; {\rm is} \; {\rm a} \; {\rm Cartier} \; {\rm divisor}\}.$$
But since $F$ is a smooth rational curve this is
$$r=\min\{m|\,m(K_X+2/dD)|_F\; {\rm has} \; {\rm integer} \; {\rm coefficients}\}=\frac{d}{2}$$
and the second part of the statement is proved.
In order to prove the third part of the statement we remark that since $D$ is smooth at $p_i$
and $f|_D$ ramifies at $p_i$ the only possibility is that $D$ is tangent to $F$ at $p_i$
with order of tangency exactly $l_i$.\\
Then we can apply Lemma \ref{3ind}
and by Equation (\ref{diesis}) an expression for $\gamma$ is
$$\gamma= 1 +\frac{1}{l_i}- \frac{2}{d}.$$
Since $l_i$ and $d$ are coprime, $l_i d$ divides $V$ for all $i$.
Again since $l_i$ and $l_j$ are coprime for all $i\neq j$
$$l_1\ldots l_N d\mid V.$$
Since $l_i\geq d/2=r$ for all $i$ we have
$$V\geq l_1\ldots l_N d\geq 2 r^{N+1}.$$
\end{proof}

\begin{proof}[Proof of Theorem \ref{mainteo} (1)]
Let $N$ be a positive integer and $f\colon X\rightarrow Z$
be a $\mathbb{P}^1$-bundle on a smooth curve.
Let $U\subseteq Z$ be an open set that trivializes the $\mathbb{P}^1$-bundle and
such that we have a local coordinate $t$ on it.
Take $d,l_1,\ldots,l_N\in\mathbb{N}$ be such that
$$l_0:=0<l_1<\ldots<l_N<l_{N+1}:=d$$
and such that they verify conditions (1)(2)(3) of Lemma \ref{spiegone}.
Let $o_1,\ldots,o_N$ be distinct points in $U$.
Let $[u:v]$ be the coordinates on the fibre and $x=u/v$
the local coordinate on the open set $\{v\neq 0\}$.
Let $D$ be the Zariski closure in $X$ of
$$D_0=\left\{\sum_{k=1}^{N+1}\left((x^{l_{k-1}}+\ldots +x^{l_{k} -1})\prod_{i=k}^N (t-o_i)\right)\right\}.$$
The restriction of $D$ to the fibre over $o_i$
is the zero locus of a polynomial of the form
$$h_i(x)=x^{l_i}q_i(x)$$
such that $x$ does not divide $q_i$.
Notice that $D$ is smooth at the points $p_i=(0,o_i)$
because the derivative with respect to $t$ of the polynomial
that defines $D_0$ is non-zero at those points.
This insures that $D$ is tangent to the fibre $F=f^{\ast}o_i$
with multiplicity exactly $l_i$
and then that $$f|_D\colon D\rightarrow Z$$
has ramification index exactely $l_i$ at $p_i$.
The fibration $f\colon (X,2/dD)\rightarrow Z$
satisfies all the hypotheses of Lemma \ref{spiegone}.
Therefore if $V$ is the minimum positive integer such that
$VM_Z$ has integer coefficients we have $V\geq r^{N+1}$.
\end{proof}

\begin{proof}[Proof of Theorem \ref{mainteo} (2)]
Let $B_Z=\sum b_i o_i$ be the discriminant divisor.
Let $V$ be the minimum integer number such that $V B_Z$ has integer coefficients.
If we write $b_i=u_i/v_i$ with $u_i,v_i\in \mathbb{N}$ and coprime it is clear that
$V=lcm\{v_i\}$.
We have seen in the proof of Theorem \ref{dim2} that
$v_i$ divides $l_i r$ for some $l_i\leq 2r$.
Then $$V=lcm\{v_i\} \mid lcm\{l_i r\}.$$
Let us remark that if $q$ is a prime number such that
$q^k$ divides $V$
then there exists a point $p$ such that $q^k$ divides $l_p r$.
Let $r=\prod q_i^{k(q_i)}$ be the decomposition of $r$ into prime factors 
and suppose that $q$ is equal to some prime $q_1$.
We have then that $$q_1 ^{k-k(q_1)} \mid l_p\leq 2r.$$
Set $$s(q)=\max\{s \mid q^s\leq 2r\}.$$
\end{proof}

The bound of Theorem \ref{mainteo} is not far from being sharp
thanks to the following example.

\begin{proof}[Proof of Theorem \ref{mainteo} (3)]
Let $r$ be an odd integer number.
Let $s(q)$
be the integer defined above.
Set $$h(q)=\max\{h\mid r\leq2^h q^{s(q)}\leq 2r\}$$
and set $$\{l_1<\ldots<l_N\}=\{2^{h(q)}q^{s(q)}|\;
q<2r,\;q\;{\rm prime}
\},$$
$$l_0=0, l_{N+1}=d=2r.$$
Consider the divisor $\bar{D}$ defined as the Zariski closure of
$$D_0=\left\{\sum_{k=1}^{N+1}\left((x^{l_{k-1}}+\ldots +x^{l_{k} -1})\prod_{i=k}^N (t-o_i)\right)\right\}.$$
Consider now $B=1/r\bar{D}$.
The fibration $f\colon (X,B)\rightarrow Z$ is lc-trivial.
Let $V$ be the minimum integer such that $VM_Z$ has integer coefficients.

Then for each $i=1\ldots N$ by Lemma \ref{3ind} we have the following expression for $\gamma_i$:
$$\gamma_i=1-\frac{2l_i-d}{l_i d}=1+\frac{r-l_i}{l_i r}.$$
For every $i$ we have $l_i=2^{h(q)}q^{s(q)}$ for a suitable $q$.
Since $r$ is odd
$$gcd\{2^{h(q)}q^{s(q)},r\}=q^{s'(q)}$$
for some $s'(q)$, then 
$$\gamma_i=1-\frac{l_i-r}{l_i r}=1+\frac{r/q^{s'(q)}-2^{h(q)}q^{s(q)-s'(q)}}{2^{h(q)}q^{s(q)-s'(q)}r}.$$
Then for all $q$ such that $q\leq 2r$
we have $$2^{h(q)}q^{s(q)-s'(q)}r|V$$
that implies that
$$lcm\{2^{h(q)}q^{s(q)-s'(q)}r\}|V.$$
But
$$lcm\{2^{h(q)}q^{s(q)-s'(q)}r\}=\frac{N(r)}{r}.$$
\end{proof}

\end{document}